\documentclass[12pt]{amsart}
\usepackage{amssymb,amsmath,amsthm}
\usepackage{enumerate}
\usepackage{mathrsfs}
\usepackage{color}
\newtheorem{theorem}{Theorem}[section]
\newtheorem{proposition}[theorem]{Proposition}

\newtheorem{corollary}[theorem]{Corollary}
\theoremstyle{definition}

\newtheorem{example}[theorem]{Example}

\newcommand{\clb}{\mathscr{B}}

\newcommand{\clh}{\mathcal{H}}
\newcommand{\clk}{\mathcal{K}}

\newcommand{\clm}{\mathcal{M}}
\newcommand{\cln}{\mathcal{N}}

\newcommand{\clr}{\mathcal{R}}
\newcommand{\cls}{\mathcal{S}}

\newcommand{\raro}{\rightarrow}

\newcommand{\vp}{\varphi}

\newcommand{\bd}{\mathbb{D}}

\title[Invariant subspaces of idempotents]{Invariant subspaces of idempotents on Hilbert spaces}

\author[Bala]{Neeru Bala}
\address{Indian Statistical Institute, Statistics and Mathematics Unit, 8th Mile, Mysore Road, Bangalore, 560 059, India}
\email{neerusingh41@gmail.com }

\author[Ghosh]{Nirupam Ghosh}
\address{Indian Statistical Institute, Statistics and Mathematics Unit, 8th Mile, Mysore Road, Bangalore, 560 059, India}
\email{nirupamghoshmath@gmail.com}

\author[Sarkar]{Jaydeb Sarkar}
\address{Indian Statistical Institute, Statistics and Mathematics Unit, 8th Mile, Mysore Road, Bangalore, 560 059, India}
\email{jaydeb@gmail.com, jay@isibang.ac.in}


\numberwithin{equation}{section}

\begin{document}

%\today

\begin{abstract}
In the setting of operators on Hilbert spaces, we prove that every quasinilpotent operator has a non-trivial closed invariant subspace if and only if every pair of idempotents with a quasinilpotent commutator has a non-trivial common closed invariant subspace. We also present a geometric characterization of invariant subspaces of idempotents and classify operators that are essentially idempotent.
\end{abstract}
\maketitle

\tableofcontents

\section{Introduction}\label{sec: intro}

This note deals with idempotents, pairs of (not necessarily commuting) idempotents, essentially idempotents, and the issue of invariant subspaces of bounded linear operators on Hilbert spaces. Recall that a bounded linear operator $T$ on a Hilbert space $\clh$ (in short, $T \in \clb(\clh)$) is said to be \textit{idempotent} if
\[
T^2 = T.
\]
If, in addition, $T^* = T$, then we say that $T$ is an orthogonal projection. Throughout this note, all Hilbert spaces are assumed to be complex and separable.

There are many reasons to study idempotent operators. However, from the present perspective, we remark that the invariant subspace problem has an affirmative solution if and only if every pair of Hilbert space idempotents has a common non-trivial closed invariant subspace (see Nordgren, Radjavi and Rosenthal \cite{NORDgren} and also see \cite{NRRR}). Recall that the celebrated invariant subspace problem asks  \cite{RADJAVI, RADJAVI ROSENTHA INV}: Does every bounded operator on an infinite-dimensional Hilbert space have a non-trivial closed invariant subspace?

From this point of view, it is also important to remark that there exist triples of idempotents without a common non-trivial closed invariant subspace (see Davis \cite{Davis}).

Not only the answer to the general invariant subspace (equivalently, the reformulation of Nordgren, Radjavi and Rosenthal as stated above) problem is unknown, but even for selective operators, like essentially self-adjoint idempotents, finite-rank perturbations of normal operators, and quasinilpotent operators, the problem is mysteriously unsettled. In fact, the study of invariant subspaces of quasinilpotent operators is still indefinite (however, see Apostol and Voiculescu \cite{Voiculescu}, Foia\c{s} and Pearcy \cite{FOIAS}, Herrero \cite{Herrero}, and the monograph \cite{RADJAVI}).

Along the lines of \cite{NORDgren}, Bernik and Radjavi \cite{BERNIK RADJAVI} proved that every essentially self-adjoint operator has a non-trivial closed invariant subspace if and only if every pair of essentially self-adjoint idempotents has a common non-trivial closed invariant subspace. Recall that $T \in \clb(\clh)$ is \textit{essentially self-adjoint} if
\[
T - T^* \in \clk(\clh),
\]
where $\clk(\clh)$ denotes the closed ideal of compact operators on $\clh$. Recall also that an operator $T \in \clb(\clh)$ is \textit{quasinilpotent} if
\[
\sigma(T) = \{0\},
\]
where $\sigma(T)$ denotes the spectrum of $T$. The quasinilpotent counterpart of Bernik and Radjavi theorem is one of our main results (see Theorem \ref{thm: quasi nil}):

\begin{theorem}
The following properties of operators on Hilbert spaces are equivalent:
\begin{enumerate}
\item Every quasinilpotent operator has a non-trivial closed invariant subspace.
\item  Every pair of idempotents with a quasinilpotent commutator has a non-trivial common closed invariant subspace.
\end{enumerate}
\end{theorem}

In the proof of the above theorem and most of the other results of this paper, we make use of geometric representations of idempotent operators \cite[Theorem 2.6]{ANDO}. To state the result, we first set up some notation. Throughout the paper, for each  $T \in \clb(\clh)$, we denote by $\clr(T)$ and $\cln(T)$ the range space and the null space of $T$, respectively. Also, we will use interchangeably the identity
\[
\cln(T^*) = \clr(T)^\perp,
\]
and given a closed subspace $\cls \subseteq \clh$, we denote by $P_{\cls}$ the orthogonal projection onto $\cls$. Now we are ready to state the geometric representation of idempotents (see Theorem \ref{thm: Ando}): Let $T \in\clb(\clh)$ be an idempotent. Then $P_{\cln(T^*)}|_{\cln(T)}: \cln(T) \raro \cln(T^*)$ is invertible, and
\[
T=\begin{bmatrix}
I_{\clr(T)}& - P_{\clr(T)}|_{\cln(T)} (P_{\cln(T^*)}|_{\cln(T)})^{-1}
\\
0&0
\end{bmatrix},
\]
on $\clh = \clr(T) \oplus \cln(T^*)$. The present version is somewhat more rigorous than that of \cite[Theorem 2.6]{ANDO}.

It is worthwhile to pause for a moment and recall invariant subspaces of idempotents on Banach spaces. Let $X$ be a Banach space, and let $T\in\clb(X)$ be an idempotent. Then invariant subspaces of $T$ are all subspaces of $X$ of the form $\clr \dotplus \cln$, where $\clr$ and $\cln$ are closed subspaces of $\clr(T)$ and $\cln(T)$, respectively (cf. \cite[Proposition 3]{ALLAN}). Evidently, if $X$ is a Hilbert space, then this representation is not (explicitly) compatible with the geometry of Hilbert spaces. From this point of view, and to some extent, Theorem \ref{thm all invariant subspace} takes care of the geometry of ambient spaces.

\begin{theorem}
Let $T\in\clb(\clh)$ be an idempotent, and let $\cls$ be a closed subspace of $\clh$. Then $\cls$ is invariant under $T$ if and only if $\cls$ is invariant under $(P_{\cln(T^*)}|_{\cln(T)})^{-1} P_{\cln(T^*)}$.
\end{theorem}

Now we turn to essentially idempotent operators. An operator $T \in \clb(\clh)$ is \textit{essentially idempotent} if $T$ is idempotent modulo the compact operators, that is
\[
T^2-T \in \clk(\clh).
\]
In Theorem \ref{prop essential idempotent}, we present a complete classification of essentially idempotent operators:

\begin{theorem}
Let $T \in \clb(\clh)$. Then $T$ is essentially idempotent if and only if one of the following holds:
\begin{enumerate}
\item $T-I\in\clk(\clh)$.
\item $T\in\clk(\clh)$.
\item $T-S\in\clk(\clh)$ for some idempotent $S \in \clb(\clh)$ such that both $\clr(S)$ and $\cln(S^*)$ are infinite-dimensional.
\end{enumerate}
\end{theorem}

We apply this to essentially commuting pairs of idempotents. A pair of bounded linear operators $(T_1, T_2)$ on $\clh$ is called \textit{essentially commuting} if
\[
T_1T_2-T_2T_1 \in \clk(\clh).
\]
We have the following (see Corollary \ref{cor: ess comm idemp}): Let $(T_1,T_2)$ be a pair of essentially commuting idempotents on $\clh$. Then one of the following holds:
\begin{enumerate}
\item $T_1+T_2-I\in\clk(\clh)$.
\item $T_1-T_2\in\clk(\clh)$.
\item There exists an idempotent $S \in \clb(\clh)$ such that $(T_1+T_2-I)S\in\clk(\clh)$.
\end{enumerate}

For the reverse direction, (1) as well as (2) implies that $(T_1,T_2)$ is essentially commuting. However, condition (3) does not necessarily imply that $(T_1,T_2)$ is essentially commuting (see Example \ref{example: conv}).

It is worthwhile to note that the commutator $D=T_1T_2-T_2T_1$ plays an important role in the problem of common invariant subspaces of a pair of idempotents $(T_1, T_2)$. Indeed, since $D^2$ commutes with both $T_1$ and $T_2$, it follows that $T_1$ and $T_2$ have a common invariant subspace whenever $D^2$ has a non-trivial hyperinvariant subspace.

The rest of this paper is divided into four sections. In Section \ref{sec: rep of idem}, we describe representations of idempotents. The main result here is due to Ando \cite{ANDO}, however, our presentation is different from the original one and in many respects appropriately customized to our needs in the latter part of the paper. Section \ref{sec:quasi nil} deals with invariant subspaces of quasinilpotent operators. Section \ref{sec: inv sub idemp} focuses on invariant subspaces of idempotents. Finally, Section \ref{sec-ess idemp} deals with classifications and representations of essentially idempotent operators.

\section{Representations of idempotents}\label{sec: rep of idem}

In this section, we collect some of the results and tools which will be used in later sections. The results are known, but our presentation is different, which is also essential for the main contribution of this paper.

We begin with a result of Ando \cite[Theorem 2.6]{ANDO} concerning representations of idempotents. Here we present a complete proof, which also takes care of the alignment of subspaces and orthogonal projections (compare the statement and the proof of \cite[Theorem 2.6]{ANDO}). Recall that for $T \in \clb(\clh)$, we have $\cln(T^*) = \clr(T)^\perp$.

\begin{theorem}\label{thm: Ando}
Let $T \in\clb(\clh)$ be an idempotent. Then $P_{\cln(T^*)}|_{\cln(T)}: \cln(T) \raro \cln(T^*)$ is invertible, and
\[
T=\begin{bmatrix}
I_{\clr(T)}& - P_{\clr(T)}|_{\cln(T)} (P_{\cln(T^*)}|_{\cln(T)})^{-1}
\\
0&0
\end{bmatrix} \text{ on } \clh = \clr(T) \oplus \cln(T^*).
\]
\end{theorem}
\begin{proof}
First note that
\begin{equation}\label{eqn: range same}
\clr(P_{\cln(T^*)} P_{\cln(T)}) = \cln(T^*).
\end{equation}
Indeed, $\clr(P_{\cln(T^*)} P_{\cln(T)}) \subseteq \cln(T^*)$ follows from the Douglas range inclusion theorem. For the other inclusion, let $x \in \cln(T^*) = \clr(I - P_{\clr(T)})$, and write $x = (I - P_{\clr(T)}) y$ for some $y\in \clh$. Write $y = y_1 + y_2$, where $y_1 \in \clr(T)$ and $y_2 \in \cln(T)$. Then
\[
x = (I - P_{\clr(T)}) (y_1 + y_2) = (I - P_{\clr(T)}) y_2 = P_{\cln(T^*)} y_2 = P_{\cln(T^*)} P_{\cln(T)} y_2,
\]
proves \eqref{eqn: range same}. Next we show that $P_{\cln(T^*)}|_{\cln(T)}: \cln(T) \raro \cln(T^*)$ is invertible. In view of \eqref{eqn: range same}, it is enough to prove that $P_{\cln(T^*)}|_{\cln(T)}$ is injective. Let
\[
P_{\cln(T^*)}x = 0,
\]
for some $x \in \cln(T)$. Then $x = P_{\clr(T)} x$, that is, $x \in \clr(T)$, and hence $x \in \cln(T) \cap \clr(T)$. Then $x = 0$ and hence $P_{\cln(T^*)}|_{\cln(T)}$ is invertible. Therefore
\[
A_T := (P_{\cln(T^*)}|_{\cln(T)})^{-1}: \cln(T^*) \raro \cln(T),
\]
is well defined. Now we turn to the block operator matrix representation of $T$ with respect to $\clh = \clr(T) \oplus \cln(T^*)$. Since $T^2 = T$, it follows that $T P_{\clr(T)} = P_{\clr(T)}$, which also implies that $P_{\cln(T^*)} T|_{\cln(T^*)} = 0$. In particular, $T|_{\clr(T)} = I_{\clr(T)}$ and hence
\[
T =\begin{bmatrix}
I_{\clr(T)}& P_{\clr(T)} T|_{\cln(T^*)}
\\
0&0
\end{bmatrix}\in \clb(\clr(T) \oplus \cln(T^*)).
\]
To complete the proof it suffices to prove that $P_{\clr(T)} T|_{\cln(T^*)} = - P_{\clr(T)}|_{\cln(T)} A_T$. We have
\[
\begin{split}
P_{\clr(T)} T (P_{\cln(T^*)} P_{\cln(T)}) & = P_{\clr(T)} (T - T P_{\clr(T)}) P_{\cln(T)}
\\
& = - P_{\clr(T)} T P_{\clr(T)} P_{\cln(T)},
\end{split}
\]
as $T P_{\cln(T)} = 0$. Again, by $T P_{\clr(T)} = P_{\clr(T)}$, we have
\[
\begin{split}
P_{\clr(T)} T (P_{\cln(T^*)} P_{\cln(T)}) & = - P_{\clr(T)} P_{\cln(T)}
\\
& = - P_{\clr(T)}|_{\cln(T)} A_T (P_{\cln(T^*)}|_{\cln(T)}) P_{\cln(T)}
\\
& = - P_{\clr(T)}|_{\cln(T)} A_T (P_{\cln(T^*)} P_{\cln(T)}).
\end{split}
\]
Then \eqref{eqn: range same} implies that $P_{\clr(T)} T|_{\cln(T^*)} = - P_{\clr(T)}|_{\cln(T)} A_T$ and completes the proof of the theorem.
\end{proof}

Particularly, we have the following similarity between an idempotent and its corresponding projection.

\begin{corollary}\label{cor: idemp proj}
If $T \in \clb(\clh)$ is an idempotent, then $T = V P_{\clr(T)} V^{-1}$, where
\[
V = \begin{bmatrix}
I_{\clr(T)}& P_{\clr(T)}|_{\cln(T)} (P_{\cln(T^*)}|_{\cln(T)})^{-1}
\\
0& I_{\cln(T^*)}
\end{bmatrix}\in \clb(\clr(T) \oplus \cln(T^*)).
\]
\end{corollary}
\begin{proof}
The proof follows at once from the representations of $P_{\clr(T)}$ and $V^{-1}$ on $\clh = \clr(T) \oplus \cln(T^*)$ as $P_{\clr(T)} = \begin{bmatrix}
I_{\clr(T)} & 0
\\
0&0
\end{bmatrix}$ and $V^{-1} = \begin{bmatrix}
I_{\clr(T)}& -P_{\clr(T)}|_{\cln(T)} (P_{\cln(T^*)}|_{\cln(T)})^{-1}
\\
0& I_{\cln(T^*)}
\end{bmatrix}$, respectively.
\end{proof}

In particular, we also have
\[
\text{Lat} T = \{V \cls: \cls \in \text{Lat} P_{\clr(T)}\},
\]
where $\text{Lat}T$ denotes the lattice of invariant subspaces of $T$. We will return to this theme in Section \ref{sec: inv sub idemp} (more specifically, see \eqref{eqn: lattice}).

\section{Idempotents with quasinilpotent commutators}\label{sec:quasi nil}

In this section, we connect the problem of invariant subspaces of quasinilpotent operators with the invariant subspaces of pairs of idempotents admitting quasinilpotent commutators. Again, recall that $T \in \clb(\clh)$ is \textit{quasinilpotent} if $\sigma(T) = \{0\}$.

We recall the following classical fact about Riesz projections \cite[Theorem 2.2, Page 10]{GOLDBERG}. Let $T\in \clb(\clh)$ and let $\sigma \subseteq \sigma(T)$ be an isolated part of $\sigma(T)$ (that is, both $\sigma$ and $\sigma(T) \setminus \sigma$ are closed subsets of $\sigma(T)$). Then the Riesz functional calculus
\[
P_\sigma = \frac{1}{2\pi i}\int_{\gamma}(\lambda-T)^{-1}d\lambda,
\]
is an idempotent, where $\gamma$ is a contour separating $\sigma$ and $\sigma(T)\setminus\sigma$. Moreover, $\clr(P_\sigma)$ and $\cln(P_\sigma)$ are invariant under $T$, and
\[
\sigma(T|_{\clr(P_\sigma)})=\sigma \text{ and } \sigma(T|_{\cln(P_\sigma)})=\sigma(T)\setminus\sigma.
\]
We also need the spectral behaviour of products of idempotents \cite[Theorem 1]{BARRAA}.

\begin{theorem}\label{thm spectra of difference}
Let $p$ and $q$ be idempotents in a Banach algebra. Then
\[
\sigma(pq)\setminus\{0,1\}=\{1-\mu^2:\mu\in\sigma(p-q)\setminus\{0, \pm 1\}\}.
\]
\end{theorem}

Now we are ready for our invariant subspace theorem for quasinilpotent operators. Throughout the proof, we will always assume that subspaces under consideration are closed.

\begin{theorem}\label{thm: quasi nil}
Every quasinilpotent operator has a non-trivial invariant subspace if and only if every pair of idempotents with a quasinilpotent commutator has a non-trivial common invariant subspace.
\end{theorem}
\begin{proof}
Suppose every quasinilpotent operator has a non-trivial invariant subspace. Fix a Hilbert space $\clh$ and suppose $T_1$ and $T_2$ be two idempotents on $\clh$. Set
\[
D = T_1 T_2 - T_2 T_1,
\]
and assume that $\sigma(D)=\{0\}$. Since $T_1$ and $T_2$ are idempotents, an easy calculation shows that (cf. \cite{ALLAN})
\begin{equation}\label{eqn D2}
(T_1-T_2)^4-(T_1-T_2)^2=(T_1T_2-T_2T_1)^2.
\end{equation}
By the spectral mapping theorem
\[
\{\lambda^4-\lambda^2:\lambda\in\sigma(T_1-T_2)\} = \{0\},
\]
and hence, $\sigma(T_1-T_2)\subseteq\{0, \pm 1\}$. Theorem \ref{thm spectra of difference} then implies that $\sigma(T_1T_2)\subseteq\{0,1\}$. In view of Theorem \ref{thm: Ando}, on $\clh = \clr(T_1) \oplus \cln(T_1^*)$, write
\[
T_1=\begin{bmatrix}
I&X_{T_1}
\\
0&0
\end{bmatrix},
\]
where $X_{T_1} = - P_{\clr(T_1)}|_{\cln(T_1)} (P_{\cln(T_1^*)}|_{\cln(T_1)})^{-1}$, and set
\[
T_2=\begin{bmatrix}
A&B
\\
C&D
\end{bmatrix}.
\]
Then
\[
T_1T_2 = \begin{bmatrix}
A+X_{T_1}C&B+X_{T_1} D
\\
0&0
\end{bmatrix},
\]
which implies $\sigma(A + X_{T_1} C) \subseteq \sigma(T_1T_2) \subseteq \{0,1\}$. We claim that $A + X_{T_1} C$ has a non-trivial invariant subspace. Indeed, if $\sigma(A + X_{T_1} C)=\{0,1\}$, then the discussion preceding the statement of this theorem implies that $A + X_{T_1} C$ has a non-trivial invariant subspace. On the other hand, $A + X_{T_1} C$ or $A + X_{T_1} C - I$ is quasinilpotent according to
\[
\sigma(A + X_{T_1} C)=\{0\} \text{ or } \{1\}.
\]
Our assumption guarantees that $A + X_{T_1} C$ or $A + X_{T_1} C- I$ has a non-trivial invariant subspace. For the latter case, if a closed subspace $\cls$ is invariant under $A + X_{T_1} C- I$, then it is easy to see that $\cls$ is also invariant under $A + X_{T_1} C$. Consequently, $A + X_{T_1} C$ has a non-trivial invariant subspace, which we denote by $\cls$. Note that $\{0\} \subsetneqq \cls \subsetneqq \clr(T_1)$.

\noindent We now proceed to check that $T_1$ and $T_2$ have a common non-trivial invariant subspace. By Corollary \ref{cor: idemp proj}, we have that $P_{\clr(T_1)} = V_1^{-1} T_1 V_1$, where
\[
V_1 = \begin{bmatrix}
I_{\clr(T_1)}& -X_{T_1}
\\
0& I_{\clr(T_1)^\perp}
\end{bmatrix}\in \clb(\clr(T_1) \oplus \cln(T_1^*)).
\]
Set
\[
Q = V_1^{-1} T_2 V_1,
\]
and consider the subspace
\[
\clm = \cls \oplus \overline{P_{\cln(T_1^*)} Q \cls}.
\]
Since $\cls \subsetneqq \clr(T_1)$, it follows that $\clm$ is a proper subspace of $\clh$. We claim that $\clm$ is invariant under $P_{\clr(T_1)}$ and $Q$. Observe that the above matrix representation of $T_1T_2$ implies
\[
A+X_{T_1}C = P_{\clr(T_1)} T_1 T_2|_{\clr(T_1)}.
\]
On the other hand, by using $T_1 = V_1 P_{\clr(T_1)} V_1^{-1}$ and $T_2 = V_1 Q V_1^{-1}$, we have
\[
T_1 T_2 = V_1 P_{\clr(T_1)} Q V_1^{-1},
\]
which along with $V_1 P_{\clr(T_1)} = P_{\clr(T_1)}$ implies that
\[
P_{\clr(T_1)} T_1 T_2|_{\clr(T_1)} = P_{\clr(T_1)}(V_1 P_{\clr(T_1)} Q V_1^{-1})|_{\clr(T_1)} = P_{\clr(T_1)} Q|_{\clr(T_1)},
\]
and hence
\[
A+X_{T_1}C = P_{\clr(T_1)} Q|_{\clr(T_1)}.
\]
Since $\cls \subsetneqq \clr(T_1)$, it follows that
\[
P_{\clr(T_1)} \clm = P_{\clr(T_1)} \cls = \cls \subseteq \clm,
\]
that is, $\clm$ is invariant under $P_{\clr(T_1)}$. Next, we observe that
\[
P_{\clr(T_1)} Q \cls = P_{\clr(T_1)} Q|_{\clr(T_1)} \cls = (A+X_{T_1}C) \cls \subseteq \cls.
\]
Note that, by the definition of $\clm$, we have
\[
(I - P_{\clr(T_1)}) Q \cls = P_{\cln(T_1^*)} Q \cls \subseteq \clm,
\]
and hence the above inclusion implies that $Q \cls \subseteq \clm$. This and the above inclusion, again, imply that
\[
\begin{split}
Q P_{\cln(T_1^*)} Q \cls & = Q (I - P_{\clr(T_1)})Q \cls
\\
& \subseteq Q \cls + Q P_{\clr(T_1)} Q \cls
\\
& \subseteq Q \cls + Q \cls
\\
& \subseteq \clm.
\end{split}
\]
Thus we have proved that $Q \cls \subseteq \clm$ and $Q P_{\cln(T_1^*)} Q \cls \subseteq \clm$, which clearly implies that $\clm$ is invariant under $Q$. Therefore, $\clm$ is a non-trivial common invariant subspace for $P_{\clr(T_1)}$ and $Q$, and consequently $V_1\clm$ is a non-trivial common invariant subspace for $T_1$ and $T_2$.

We now turn to the converse. Assume that every pair of idempotents with a quasinilpotent commutator has a non-trivial common invariant subspace. Let $A\in\clb(\clh)$ be a quasinilpotent operator. Our aim is to prove that $A$ has a non-trivial invariant subspace. Following \cite{NORDgren}, define
\begin{align*}
T_1=\begin{bmatrix}
A&A
\\
I-A&I-A
\end{bmatrix}\text{ and }T_2=\begin{bmatrix}
I&0\\
0&0
\end{bmatrix}.
\end{align*}
Then \[
T_1T_2=\begin{bmatrix}
	A&0\\
	I-A&0
\end{bmatrix}.\]
Since $\sigma(T_1T_2)=\sigma(A)\cup\{0\}$, we have $\sigma(T_1T_2)=\{0\}$. By Theorem \ref{thm spectra of difference}, we know that
\[
\sigma(T_1-T_2)\subseteq\{0,\pm 1\}.
\]
We also know (see \eqref{eqn D2}) that
\[
D^2=(T_1T_2-T_2T_1)^2=(T_1-T_2)^4-(T_1-T_2)^2.
\]
By the spectral mapping theorem, we have $\sigma(D)=\{0\}$. Therefore, by assumption, $T_1$ and $T_2$ have a common non-trivial invariant subspace, say $\cls \subsetneqq \clh \oplus \clh$. Then
\[
\cls = T_2 \cls \dotplus (I-T_2)\cls.
\]
The rest of the proof follows a similar line as in the proof of the theorem in \cite[page 66]{NORDgren}. In the present case, also, the algebra generated by $I$, $T_1$, and $T_2$ contains
\[
\begin{bmatrix}
0&A
\\
0&0
\end{bmatrix} \text{ and }\begin{bmatrix}
0&0\\
I-A&0
\end{bmatrix}.
\]
Now, we claim that $T_2\cls$ and $(I-T_2)\cls$ are proper subspaces of $\clh$. Without loss of generality, we assume $0$ is in the continuous spectrum of $A$, that is $A$ and $A^*$ are injective but $\clr(A)$ is not closed (indeed, in other cases, $A$ will have a non-trivial invariant subspace). If $T_2 \cls =\clh$, then $x\oplus (I-A)y \in \cls$ for all $x,y \in \clh$. Since $(I-A)$ is invertible, it follows that $\cls = \clh\oplus\clh$. Similarly, if $(I-T_2)\cls = \clh$, then $Ax\oplus y\in \cls$ for all $x,y\in\clh$, would imply that $\cls = \clh\oplus\clh$, as $\overline{\clr(A)}=\clh$. Since $\cls$ is non-trivial, either $T_2\cls$ or $(I-T_2)\cls$ is nonzero. Hence either $T_2\cls$ or $(I-T_2)\cls$ is a non-trivial invariant subspace for $A$. This completes the proof of the theorem.
\end{proof}

In the above proof, the commutator $T_1T_2-T_2T_1$ plays an important role. As already pointed out in Section \ref{sec: intro}, commutators are in general useful in the study of the joint invariant subspaces of pairs of idempotents \cite{ALLAN}. In Section \ref{sec-ess idemp}, we also study representations of essentially commuting pairs of idempotents.

\section{Invariant subspaces of idempotents}\label{sec: inv sub idemp}

As we have already pointed out in Section \ref{sec: intro}, the invariant subspaces of an idempotent $T$ acting on some Banach space $X$ are all closed subspaces of $X$ of the form $\clr \dotplus \cln$, where $\clr$ is a closed subspace of $\clr(T)$ and $\cln$ is a closed subspace of $\cln(T)$. However, in the setting of Hilbert spaces, this representation does not in general carry all the geometry of the ambient spaces. In the following, we examine how the geometry of the Hilbert space method fits into the picture of invariant subspaces of idempotents.

\begin{theorem}\label{thm all invariant subspace}
Let $T\in\clb(\clh)$ be an idempotent, and let $\cls$ be a closed subspace of $\clh$. Then $\cls$ is invariant under $T$ if and only if $\cls$ is invariant under $(P_{\cln(T^*)}|_{\cln(T)})^{-1} P_{\cln(T^*)}$.
\end{theorem}
\begin{proof}
Let $x \in \cls$ and write $x = x_1 \oplus x_2 \in \clr(T) \oplus \cln(T^*) = \clh$. By Theorem \ref{thm: Ando}, we have
\[
Tx = x_1 - P_{\clr(T)}|_{\cln(T)} A_T x_2 \in \clr(T),
\]
where $A_T = (P_{\cln(T^*)}|_{\cln(T)})^{-1}: \cln(T^*) \raro \cln(T)$. We decompose $x_2$ further as
\[
x_2 = u_1 \dotplus u_2 \in \clr(T) \dotplus \cln(T).
\]
Since
\[
P_{\clr(T)} u_2 = P_{\clr(T)} (x_2 - u_1) = - P_{\clr(T)} u_1 = -u_1,
\]
it follows that
\[
P_{\cln(T^*)}|_{\cln(T)} u_2 = P_{\clr(T)^\perp} u_2 = u_2 - P_{\clr(T)} u_2 = u_2 + u_1 = x_2,
\]
and hence $A_T x_2 = (P_{\cln(T^*)}|_{\cln(T)})^{-1} x_2 = u_2$. Therefore
\[
P_{\clr(T)}|_{\cln(T)} A_T x_2 = P_{\clr(T)} u_2 = P_{\clr(T)}(x_2 - u_1) = -u_1,
\]
that is, $P_{\clr(T)}|_{\cln(T)} A_T x_2 = -u_1$. Then
\[
T x = x_1 + u_1 = x_1 + x_2 - u_2 = x - u_2,
\]
from which we deduce
\[
(A_T P_{\cln(T^*)}) x = u_2 = x - Tx.
\]
Now, if we assume that $\cls$ is invariant under $T$, then $u_2 = x - Tx \in \cls$ implies that $\cls$ is invariant under $A_T P_{\clr(T)^\perp}$.	Conversely, suppose that $(A_T P_{\cln(T^*)}) \cls \subseteq \cls$, and suppose $x \in \cls$. Then from the above equality, it follows that
\[
Tx = x - (A_T P_{\cln(T^*)}) x \in \cls,
\]
that is, $\cls$ is invariant under $T$, which completes the proof of the theorem.
\end{proof}

In particular, if $T \in \clb(\clh)$ is an idempotent, then
\begin{equation}\label{eqn: lattice}
\text{Lat} T = \text{Lat} ((P_{\cln(T^*)}|_{\cln(T)})^{-1} P_{\cln(T^*)}).
\end{equation}

The following simple example illustrates the basic idea of the above theorem.

\begin{example}
Let $\varphi: \bd \rightarrow \bd$ be a holomorphic self-map of the unit disc $\bd = \{z \in \mathbb{C}: |z| < 1\}$. We consider the composition operator $C_{\vp}: H^2(\bd) \rightarrow H^2(\bd)$ defined by
\[
C_{\vp}f = f\circ\vp \qquad (f \in H^2(\bd)),
\]
where $H^2(\bd)$ denotes the Hardy space over $\bd$. It is easy to see that $C_{\vp}$ is an idempotent if and only if either $\vp$ is constant or $\vp$ is the identity map. If $\vp$ is the identity map, then $C_{\vp}$ is also the identity map and in this case, any closed subspace of $H^2(\bd)$ is an invariant subspace for $C_{\vp}$. Next we consider the case where $\vp$ is constant. Suppose $\vp \equiv \alpha$ for some $\alpha \in \bd$. Then, $\clr(C_\vp) = \mathbb{C}$ the one-dimensional space of all constant functions in $H^2(\bd)$, and
\[
\cln(C_\vp) = (z-\alpha) H^2(\bd).
\]
In particular, $\cln(C_\vp^*) = \clr(C_\vp)^\perp = z H^2(\bd)$. For each $f\in H^2(\bd)$, it follows that
\[
P_{\cln(C_\vp^*)} f = (I-P_{\clr(C_\vp)})f=f-f(0),
\]
and hence, for each $g \in H^2(\bd)$ and $a \in \mathbb{C}$, we have
\[
P_{\cln(C_\vp^*)} P_{\cln(C_\vp)}\Big((z-\alpha)(a + zg) \Big) = z(a +(z- \alpha)g).
\]
Therefore
\[
(P_{\cln(C_\vp^*)}|_{\cln(C_\vp)})^{-1} P_{\cln(C_\vp^*)} f = f - f(\alpha) \qquad (f \in H^2(\bd)),
\]
from which it immediately follows that a closed subspace $\cls \subseteq H^2(\bd)$ is invariant under $C_\vp$ if and only if $f - f(\alpha) \in \cls$ for all $f \in \cls$. In particular, $\text{span}\{1,z\}$ is invariant under $C_{\vp}$ but $\text{span}\{z\}$ is not.	
\end{example}

Let $T\in\clb(\clh)$ be an idempotent, and let $\cls \subseteq \cln(T^*)$ be a closed subspace. If $\cls \subseteq \cln(T)$, then $T \cls = \{0\}$, and hence $\cls$ is invariant under $T$. The converse is also true: 

\begin{corollary}\label{Prop invariant subspace}
Let $T\in\clb(\clh)$ be an idempotent, and let $\cls$ be a closed subspace of $\cln(T^*)$. Then $\cls$ is invariant under $T$ if and only if $\cls \subseteq \cln(T)$.
\end{corollary}
\begin{proof}
Suppose $\cls$ is invariant under $T$. Since $\cls \subseteq \cln(T^*)$, by Theorem \ref{thm all invariant subspace}, we have
\[
(P_{\cln(T^*)}|_{\cln(T)})^{-1} \cls \subseteq \cls.
\]
Moreover, since
\[
\clr((P_{\cln(T^*)}|_{\cln(T)})^{-1}) = \cln(T),
\]
for each $x \in \cls$, we have
\[
y: = (P_{\cln(T^*)}|_{\cln(T)})^{-1}x \in \cln(T) \cap \cls \subseteq \cln(T) \cap \cln(T^*),
\]
which implies
\[
P_{\cln(T^*)} x = x = (P_{\cln(T^*)}|_{\cln(T)}) y = y.
\]
Therefore $x = y \in \cln(T) \cap \cln(T^*)$, which proves that $\cls \subseteq \cln(T) \cap \cln(T^*)$ and completes the proof of the corollary.
\end{proof}

In particular, if $T\in\clb(\clh)$ is an idempotent, then a closed subspace $\cls \subseteq \cln(T^*)$ is invariant under $T$ if and only if $T$ reduces $\cls$. Moreover, we have the following:

\begin{proposition}
Let $T\in\clb(\clh)$ be an idempotent, and let $\cls \subseteq \clr(T)$ be a closed subspace. Then $\cls$ reduces $T$ if and only if $\cls \subseteq \clr(T) \cap \clr(T^*)$.
\end{proposition}
\begin{proof}
Suppose $\cls \subseteq \clr(T)$ reduces $T$. By Theorem \ref{thm: Ando}, on $\clh = \clr(T) \oplus \cln(T^*)$, we have
\[
T^* =\begin{bmatrix}
I_{\clr(T)}& 0
\\
X_T^*&0
\end{bmatrix},
\]
where $X_T =  - P_{\clr(T)}|_{\cln(T)} (P_{\cln(T^*)}|_{\cln(T)})^{-1}$. Fix $x \in \cls$. Since $\cls$ is invariant under $T^*$, it follows that
\[
X_T^*x = -((P_{\cln(T^*)}|_{\cln(T)})^*)^{-1} P_{\clr(T)} x = 0.
\]
Since $\cls \subseteq \clr(T)$, we have
\[
((P_{\cln(T^*)}|_{\cln(T)})^*)^{-1} x = 0,
\]
and hence $P_{\cln(T)} x = 0$. Therefore, $\cls \subseteq \cln(T)^\perp$, and hence $\cls \subseteq \clr(T)\cap \clr(T^*)$. The converse can be seen easily by the fact that $\clr(T)\cap \clr(T^*)$ is invariant under both $T$ and $T^*$.
\end{proof}

\section{Essentially idempotent operators}\label{sec-ess idemp}

In this section, we classify essentially idempotent operators. Further, we provide some basic descriptions of essentially commuting pairs of idempotents. Throughout, $\clk(\clh)$ will denote the ideal of compact operators in $\clb(\clh)$. The Calkin algebra is the quotient space $\clb(\clh)/\clk(\clh)$, the algebra of bounded linear operators modulo the compacts. Also, we denote by $\pi:\clb(\clh)\rightarrow \clb(\clh)/\clk(\clh)$ the canonical quotient map. The essential spectrum $\sigma_e(T)$ of $T \in \clb(\clh)$ is defined by $\sigma_e(T) = \sigma(\pi(T))$. It is well known that
\[
\sigma_e(T) = \{\lambda \in \mathbb{C}: T - \lambda I_{\clh} \text{ not Fredholm}\}.
\]

We begin with essentially idempotent operators. Recall that an operator $T \in \clb(\clh)$ is \textit{polynomially compact} if there exists a polynomial $p \in \mathbb{C}[z]$ such that $p(T) \in \clk(\clh)$. If $p(T) \in \clk(\clh)$ for $p=z - z^2$, then we say that $T$ is \textit{essentially idempotent}. Equivalently, an operator $T \in \clb(\clh)$ is essentially idempotent if
\[
T - T^2 \in \clk(\clh).
\]
The following classification of polynomially compact operators is due to Olsen \cite[Theorem 2.4]{OLSEN}: Let $T \in \clb(\clh)$, and suppose $p(T) \in \clk(\clh)$ for some $p \in \mathbb{C}[z]$. Then there exists a compact operator $K \in \clk(\clh)$ such that
\[
p(T+K) = 0.
\]
We are now ready for the classification of essentially idempotent operators.

\begin{theorem}\label{prop essential idempotent}
Let $T \in \clb(\clh)$. Then $T$ is essentially idempotent if and only if one of the following holds:
\begin{enumerate}
\item $T-I\in\clk(\clh)$.
\item $T\in\clk(\clh)$.
\item $T-S\in\clk(\clh)$ for some idempotent $S \in \clb(\clh)$ such that both $\clr(S)$ and $\cln(S^*)$ are infinite-dimensional.
\end{enumerate}
\end{theorem}
\begin{proof}
If $T$ satisfies any of the conditions (1) through (3), then it is easy to see that $T$ is essentially idempotent. Thus we only need to prove the converse. Let $T^2-T\in\clk(\clh)$. By Olsen, as above, there exists $K \in\clk(\clh)$ such that
\[
(T+K)^2-(T+K)=0.
\]
Let $S:=T+K$, so that $S^2=S$.	Assume that $S$ is a projection, that is, $S=S^*$. Then
\[
\begin{split}
T^*T-TT^* & = (S - K^*) (S - K) - (S - K) (S - K^*)
\\
& = S^2 - S^2 + \text{ compact},
\end{split}.
\]
Therefore $T^*T-TT^* \in\clk(\clh)$, which implies that $\pi(T^*T-TT^*) = 0$, that is, $\pi(T)$ is normal. Also note that
\[
\sigma_{e}(T)=\sigma_e (S)\subseteq \sigma(S) \subseteq \{0,1\}.
\]
We have three cases to consider: Suppose $\sigma_e(T) = \{0\}$. Then $\sigma(\pi(T)) = \{0\}$. Hence by the spectral radius formula for normal elements of $C^*$-algebras, it follows that $\pi(T) = 0$ and consequently $T \in \clk(\clh)$.

\noindent Assume now that $\sigma_e(T) = \{1\}$. Then $\sigma(\pi(T - I)) = \{0\}$. Since $\pi(T - I)$ is normal in $\clb(\clh)/\clk(\clh)$, from the argument above we must have that $\pi(T - I) = 0$, and hence $T - I \in \clk(\clh)$.

\noindent Finally, assume that $\sigma_e(T)=\{0,1\}$. Then $\sigma_e(S)=\{0,1\}$. But $S$ is a projection, and hence both $\cln(S^*)$ and $\clr(S)$ are infinite dimensional. Moreover, $S=T+K$ implies that $T-S \in \clk(\clh)$.

\noindent Now we consider the case where $S$ is not necessarily self-adjoint. By Corollary \ref{cor: idemp proj}, we have that $S = V P_{\clr(S)} V^{-1}$, where
\[
V = \begin{bmatrix}
I_{\clr(S)}& P_{\clr(S)}|_{\cln(S)} (P_{\cln(S^*)}|_{\cln(S)})^{-1}
\\
0& I_{\cln(S^*)}
\end{bmatrix}\in \clb(\clr(S) \oplus \cln(S^*)).
\]
Since $S = T+ K_1$, there exists $\tilde{K}\in\clk(\clh)$ such that
\[
P_{\clr(S)} = V^{-1}TV+\tilde{K}.
\]
Then, we know from the above argument that one of the following holds:
\begin{enumerate}
	\item $V^{-1}TV-I\in\clk(\clh)$,
	\item $V^{-1}TV\in\clk(\clh)$, or
	\item $V^{-1}TV- P_{\clr(S)} \in \clk(\clh)$ and both $\clr(S)$ and $\cln(S^*)$ are infinite-dimensional,
\end{enumerate}
which is equivalent to saying that $T-I\in\clk(\clh)$, or $T\in\clk(\clh)$, or there exists an idempotent $S_1 \in \clb(\clh)$ such that $T-S_1\in\clk(\clh)$ and both $\clr(S_1)$ and $\cln(S_1^*)$ are infinite-dimensional. This completes the proof of the theorem.
\end{proof}

We apply the above theorem to representations of essentially commuting pairs of idempotents. Recall that a pair of bounded linear operators $(T_1, T_2)$ on $\clh$ is called essentially commuting if $T_1T_2-T_2T_1 \in \clk(\clh)$.

\begin{corollary}\label{thm compact commutator}
Let $(T_1,T_2)$ be a pair of essentially commuting idempotents on $\clh$. Then one of the following holds:
\begin{enumerate}
\item $T_1T_2,\,T_2T_1\in\clk(\clh)$.
\item $T_1(I-T_2)$ and $T_2(I-T_1)$ are in $\clk(\clh)$.
\item There exists an idempotent $S\in\clb(\clh)$ such that $T_1T_2S,\,T_2T_1S\in\clk(\clh)$.
\end{enumerate}
\end{corollary}
\begin{proof}
Since $T_1$ and $T_2$ are idempotents, by \eqref{eqn D2}, we know that
\[
D^2=(T_1-T_2)^4-(T_1-T_2)^2,
\]
where $D = T_1 T_2 - T_2 T_1$. By assumption, $D \in \clk(\clh)$, and hence $R^2-R\in\clk(\clh)$, where
\[
R:=(T_1-T_2)^2.
\]
We know by Theorem \ref{prop essential idempotent} that $R-I$ or $R$ or $R-S$ belongs to $\clk(\clh)$, where $S$ is an idempotent on $\clh$. Let us first assume that $R-I\in\clk(\clh)$. Then
\[
(T_1-T_2)^2-I\in\clk(\clh),
\]
or, equivalently
\[
T_1+T_2-T_1T_2-T_2T_1-I\in\clk(\clh).
\]
Since $D\in\clk(\clh)$, we have necessarily
\[
T_1+T_2-2T_1T_2-I\in\clk(\clh).
\]
Multiplying by $T_1$ on the left, we deduce that $T_1T_2\in\clk(\clh)$. Next, assume that $R\in\clk(\clh)$, that is
\[
T_1+T_2 - T_1T_2 - T_2 T_1 \in\clk(\clh).
\]
By substituting the compact operator $D$ in this expression, we get
\[
T_1+T_2-2T_1T_2\in\clk(\clh).
\]
Again, multiplying by $T_1$ on the left, we conclude that $T_1(I-T_2)\in\clk(\clh)$. Similarly, multiplying by $T_2$ on the right, we find $T_2(I-T_1) \in\clk(\clh)$.

\noindent Finally, if $R-S\in\clk(\clh)$ for some idempotent $S \in \clb(\clh)$, then $T_1+T_2-2T_1T_2-S\in\clk(\clh)$. Multiplying by $T_1$ on the left, we get $T_1-T_1T_2-T_1S\in\clk(\clh)$. Using $D\in\clk(\clh)$, it follows that
\[
T_1-T_2T_1-T_1S\in\clk(\clh).
\]
Now multiplying by $T_2$ from the left, we have $T_2T_1S\in\clk(\clh)$.
\end{proof}

Evidently, the first conclusion is covered by the third. Moreover, a closer look at the proof reveals the following: $R - I \in \clk(\clh)$ implies that
\[
(T_1-T_2)^2-I,\,T_1T_2,\,T_2T_1\in\clk(\clh),
\]
and $R \in \clk(\clh)$ yields
\[
(T_1-T_2)^2,\,T_1(I-T_2),\,T_2(I-T_1)\in\clk(\clh),
\]
and finally, $R - S \in \clk(\clh)$ implies
\[
(T_1-T_2)^2-S,\,T_1T_2S,\,T_2T_1S\in\clk(\clh),
\]
for some idempotent $S\in\clb(\clh)$. The above three conditions yield the following corollary:

\begin{corollary}\label{cor: ess comm idemp}
Let $(T_1,T_2)$ be a pair of essentially commuting idempotents on $\clh$. Then one of the following holds:
\begin{enumerate}
\item $T_1+T_2-I\in\clk(\clh)$,
\item $T_1-T_2\in\clk(\clh)$,
\item there exists an idempotent $S \in \clb(\clh)$ such that $(T_1+T_2-I)S\in\clk(\clh)$.
\end{enumerate}
\end{corollary}

The reverse directions of Corollaries \ref{thm compact commutator} and \ref{cor: ess comm idemp} are not true in general. For instance:

\begin{example}\label{example: conv}
Consider two noncompact operators $A_1$ and $A_2$ on $\clh$, and define idempotents $T_1$, $T_2$ and $S$ on $\tilde{\clh} := \clh \oplus \clh \oplus\clh$ by
\[
T_1 =
\begin{bmatrix}
I&0&A_1
\\
0&I&A_2
\\
0&0&0
\end{bmatrix}
\text{ and }
T_2=\begin{bmatrix}
0&0&0
\\
0&0&0
\\
0&0&I
\end{bmatrix},
\]
and
\[
S=\begin{bmatrix}
I&0&0
\\
0&0&0
\\
0&0&0
\end{bmatrix}.
\]
By a simple computation, it follows that
\[
T_1 T_2 S = T_2 T_1 S =  (T_1+T_2-I)S=0\in\clk(\tilde \clh).
\]
In particular, we have $T_1 T_2 S = T_2 T_1 S \in\clk(\tilde{\clh})$, whereas, on the other hand, we have
\[
T_1 T_2 - T_2 T_1 =
\begin{bmatrix}
0&0&A_1
\\
0&0&A_2
\\
0&0&0
\end{bmatrix}\notin\clk(\tilde{\clh}).
\]
Therefore, condition (3) in Corollary \ref{thm compact commutator}, as well as in Corollary \ref{cor: ess comm idemp}, does not imply that $T_1 T_2 - T_2 T_1$ is compact.
\end{example}

On the other extreme, condition (1) or (2) of Corollary \ref{cor: ess comm idemp} implies that the commutator is essentially compact: Let $(T_1, T_2)$ be a pair of idempotents on $\clh$. Suppose
\[
T_1+T_2-I\in\clk(\clh).
\]
Then by multiplying $T_1+T_2-I$ by $T_1$ and $T_2$ on the left, respectively, it follows that $T_1T_2,\,T_2T_1\in\clk(\clh)$, which immediately implies that $T_1 T_2 - T_2 T_1 \in \clk(\clh)$. Next, we assume that
\[
T_1-T_2\in\clk(\clh).
\]
Multiplying this on the left with $T_1$ and $T_2$, respectively, we have
\[
T_1-T_1T_2,\,T_2T_1-T_2 \in \clk(\clh).
\]
Then
\[
T_1 T_2 - T_2 T_1 = -(T_1-T_1T_2)-(T_2T_1-T_2)+(T_1-T_2)\in\clk(\clh),
\]
that is, $T_1 T_2 - T_2 T_1 \in\clk(\clh)$. Therefore, in either case, the pair $(T_1, T_2)$ is essentially commuting.

Finally, in the context of difference of two idempotents and part (2) of Corollary \ref{cor: ess comm idemp}, we refer the reader to \cite{Hartwig} and \cite{Wu P}.

\smallskip

\noindent\textbf{Acknowledgement:}
The research of first author is supported by the Theoretical Statistics and Mathematics Unit,  Indian Statistical Institute, Bangalore, India, and the research of the second author is supported by the NBHM post-doctoral fellowship, Department of Atomic Energy (DAE), Government of India (File No: 0204/16(20)/2020/R\&D-II/10). The third named author is supported in part by the Core Research Grant (CRG/2019/000908), by SERB (DST), Government of India.

\end{document}